\documentclass[a4paper,12pt]{article}
\usepackage{latexsym}
\usepackage{amsmath}
\usepackage{amssymb}
\usepackage{enumerate}
\usepackage{color}
\definecolor{blue}{rgb}{0,0,1}
\definecolor{red}{rgb}{1,0,0}

\usepackage{theorem}
\newtheorem{theorem}{Theorem}[section]
\newtheorem{proposition}[theorem]{Proposition}
\newtheorem{lemma}[theorem]{Lemma}
\newtheorem{corollary}[theorem]{Corollary}

\theorembodyfont{\rmfamily}
\newtheorem{proof}{\textmd{\textit{Proof.}}}

\newtheorem{remark}[theorem]{Remark}
\newtheorem{example}[theorem]{Example}

\makeatletter

\@addtoreset{equation}{section}
\makeatother

\newcommand{\qedd}{\hfill \Box}

\newcommand{\N}{\ensuremath{\mathbb{N}}}
\newcommand{\R}{\ensuremath{\mathbb{R}}}

\title{Metric structures associated to Finsler metrics
\footnote{
Mathematics Subject Classification (2010)\,:\,53C60, 53C22.}
}
\author{Sorin V. SABAU, Kazuhiro SHIBUYA and Hideo SHIMADA}
\date{}
\begin{document}


\maketitle

\begin{abstract}
We investigate the relation between weighted quasi-metric Spaces and Finsler Spaces. We show that the induced metric of a Randers space with reversible geodesics is a weighted quasi-metric space.

\end{abstract}

\section{Introduction and Motivation}

Riemannian spaces can be represented as metric spaces. Indeed, for a Riemannian space $(M,a)$ we can define the induced metric space $(M,d_\alpha)$, with the metric 
\begin{equation}
d_\alpha:M\times M\to [0,\infty),\quad 
d_\alpha(x,y):=\inf_{\gamma\in \Gamma_{xy}}\int_a^b\alpha(\gamma(t),\dot\gamma(t))dt,
\end{equation}
where $\Gamma_{xy}:=\{ \gamma:[a,b]\to M\ |\ \gamma \textrm{ (piecewise) C}^\infty \textrm{-curve},\gamma(a)=x,\gamma(b)=y\}$ is the set of curves joining points $x$ and $y$, $\dot\gamma(t):=\frac{d\gamma(t)}{dt}$ the tangent vector to $\gamma$ at $\gamma(t)$, and $\alpha(x,X)$ the Riemannian norm of the vector $X\in T_xM$. It is easy to see that $d_\alpha$ is a metric on $M$, i.e. it satisfies the axioms:
\begin{enumerate}
\item Positiveness: $d_\alpha(x,y)>0$ if $x\neq y$, $d_\alpha(x,x)=0$,
\item Symmetry: $d_\alpha(x,y)=d_\alpha(y,x)$,
\item Triangle inequality: $d_\alpha(x,y)\leq d_\alpha(x,z)+d_\alpha(z,y)$,
\end{enumerate}
for any $x,y,z\in M$. 

More general structures than Riemannian ones are Finsler structures (see \cite{BCS00}, \cite{S01}, \cite{MHSS01} for definitions).

Similarly with the Riemannian case, one can define the induced metric of a Finsler space $(M,F)$ by 
\begin{equation}\label{Finslerian distance}
d_F:M\times M\to [0,\infty),\quad 
d_F(x,y):=\inf_{\gamma\in \Gamma_{xy}}\int_a^bF(\gamma(t),\dot\gamma(t))dt,
\end{equation}
but in this case, unlike the Riemannian counterpart, $d_F$ lacks the Symmetry condition 3 above. In fact $d_F$ is a special case of {\it quasi-metric space}.

We recall here that a  {\it quasi-metric} $d$ on a set $X$ is a function $d:X\times X\to [0,\infty)$ that  satisfies the axioms:
\begin{enumerate}
\item Positiveness: $d(x,y)>0$ if $x\neq y$, $d(x,x)=0$,
\item Triangle inequality: $d(x,y)\leq d(x,z)+d(z,y)$,
\item Separation axiom: $d(x,y)=d(y,x)=0\quad \Rightarrow \quad x=y,$
\end{enumerate}
for any $x,y,z\in X$. 

\begin{remark}
 Remark that in the definition of quasi-metric spaces, it is commonly used $d_F(x,y)=0 \Rightarrow x=y$, without assuming that both $d_F(x,y)$ and $d_F(y,x)$ are zero (Def 2.1 in \cite{JLP}). 
 This guarantees that the distance is zero only in the diagonal. Our definition here is stronger than the usual one. In general, the distance associated to a Finsler metric is a generalized metric, namely, a quasi-metric such that the forward and backward topology coincide, see Remark 2.2 in \cite{JLP}.
\end{remark}

A quasi-metric that satisfies the symmetry axiom $d(x,y)=d(y,x)$ for all $x,y\in M$ is a metric on $M$. One can easily see that this happens in the case of Riemannian and absolute homogeneous Finsler metrics, i.e. Finsler norms $F$ for which $F(x,y)=F(x,-y)$, for all $(x,y)\in TM\setminus\{0\}$.

Finsler manifolds have a richer geometrical structure than Riemannian ones. Indeed, for a given Finsler manifold $(M,F)$ we can define the {\it reverse} Finsler structure $(M,\bar F)$, where
$\bar F(x,y):=F(x,-y)$. This means that on $M$ we obtain the reverse (or the dual) quasi-metric $d_{\bar F}(x,y):=d_F(y,x)$. Moreover, from the metric point of view we can define 
\begin{itemize}
\item the symmetrization of $d_F$, namely
\begin{equation}\label{symmetrized dist}
\rho(x,y):=\dfrac{d_F(x,y)+d_F(y,x)}{2},
\end{equation}
\item the max-metric
\begin{equation}\label{max metric}
d^*(x,y):=\max\{d_F(x,y),d_F(y,x)\},
\end{equation}
\end{itemize}
for any $x,y\in M$. One can easily see that these are metrics on $M$. We also recall that because of the lack of symmetry of distance function $d_F$ it is customary to make distinction between balls, neighborhoods, etc. by calling them {\it forward} or {\it backward}, respectively. 

One special class of quasi-metric spaces are the so called {\it weighted quasi-metric spaces} $(M,d,w)$, namely $d$ is a quasi-metric on $M$ for each there exists a function $w:M\to [0,\infty)$, called the {\it weight} of $d$, that satisfies
\begin{enumerate}[\ \ 4.]
\item Weightability:  $d(x,y)+w(x)=d(y,x)+w(y),\quad \forall x,y\in M.$
\end{enumerate}
In the case the weight function $w$ is $\R$-valued, $w$ is called {\it generalized weight}. 

\begin{remark}
If $(M,d)$ is a metric space, then it can be regarded as a weighted metric space with a weight function $w=constant$. 
\end{remark}

The weighted quasi-metric spaces were initially introduced in the context of theoretical computer science \cite{M94} and their topological properties are extensively studied (see \cite{KV94} and references herein). It is worth mentioning that the study of denotational semantics of programming languages imposes a topological model defined on a weighted quasi-metric space. From topological point of view, a topological $T_0$-space $X$ admits a weighted quasi-metric provided it has a base $\bigcup_{n\in \N}\mathcal B_n$ for its topology with the property that for each $n\in \N$ there is an $m_n\in \N$ such that each point of $X$ belong to at most $m_n$ elements on $\mathcal B_n$ (\cite{KV94}). We recall that a topological space is a $T_0$-space or Kolmogorov space if for every pair of distinct points of $X$, at least one of them has an open neighborhood not containing the other. This condition is one of the separation axioms in topology and its intuitive meaning is that the points of $X$ are topologically distinguishable. 

More recently, it has been shown (\cite{SY09}) that weighted quasi-metric spaces are essential for sequence comparison in molecular biology and bioinformatics. The comparison of biological sequences (especially proteins) is the fundamental method for the investigations of the origin and function of peptide fragments with evolutionary conserved sequence. The primary method used here is the similarity search: find similar fragments to a given query amino acids sequence and infer the function refereeing to the known functions of search results. The representative tool for similarity search is BLAST that can be accessed from NCBI, DDJ or ENSEMBL web sites. 

It is also known that similarity search of biological sequences can be geometrically formalized by defining a sort of ``distance '' on the free monoid $\Sigma^{*}$ over a non-empty finite set $\Sigma$. Concretely, in the case of peptide fragments comparison, $\Sigma$ is the set of proteinogenic amino acids (arginine, histidine, lysine; aspartic acid, glutamic acid; serine, threonine, asparagine, glutamine; cysteine, glycine, proline; alanine, isoleucine, leucine, methionine, phenylalanine, tryptophan, tyrosine, valine), where each amino acid is denoted by a letter and $\Sigma^{*}$ contains all finite sequences of zero or more elements from $\Sigma$. 

Remarkably, the evolutionary distance induced by local or global alignments of peptide fragments is actually a weighted quasi-metric on $\Sigma^{*}$. Therefore, the study of sequence comparison reduces to the geometry of weighted quasi-metrics, where minimizing similarities between peptide fragments is equivalent to minimizing weighted quasi-distances between elements of the free monoid $\Sigma^{*}$ (\cite{SY09}).

However, despite of extensive investigations of weighted quasi-metric spaces from topological point of view and different applications, a study of these spaces from differential geometry point of view cannot be found in literature. 

In the present paper we will show that the metric structure induced by a Finsler metric with reversible geodesics is actually a weighted quasi-metric. This clarifies the geometrical meaning of weighted quasi-structures. 

Moreover, we obtain several interesting geometrical properties of Finsler metrics with reversible geodesics and weighted quasi-metric spaces.
\bigskip

{\bf Acknowledgements.} The authors are extremely grateful to the referee for several useful suggestions that have considerable improved the paper.
We also thanks to Miguel Angel Javaloyes for pointing out some inaccuracies in a preliminary version. 


\section{Finsler metrics and weighted quasi-metrics}

Recall that a Finsler metric $F$ on a $n$-dimensional differential manifold $M$ is called {\it with reversible geodesics} if and only if for any geodesic $\gamma:[0,1]\to M$ of $F$, the reverse curve 
$\bar\gamma(t):=\gamma(1-t)$ is also a geodesic of $F$.

We point out that even a Finsler space is with reversible geodesics, the Finslerian distance function $d_{F}$ is not symmetric, except for the absolute homogeneous case. 

We have (see \cite{MSS10}, \cite{MSS13}, \cite{SS12} and references herein)
\begin{proposition}\label{rev_geod}
Let $(M,F=F_{0}+\beta)$ be a Finsler space whose fundamental function is obtained by  a Randers change of an absolute homogeneous Finsler metric $F_{0}$ by a one-form $\beta$. Then $(M,F)$ is with reversible geodesics if and only $\beta$ is closed. 
\end{proposition}

The intuitive meaning of the Randers change 
\begin{equation}\label{Randers change}
F=F_{0}+\beta,\quad d\beta=0
\end{equation}
 is that the $F$-geodesics coincide with the $F_{0}$-geodesics as set of points, i.e. $F$ and $F_{0}$ are projectively equivalent. Remark that this Randers change is a special Randers change with $\beta$ closed. Hereafter, when referring to this formula, we always mean {\it Randers change with $\beta$ closed}.

\begin{remark}
\begin{enumerate}
\item A special case is the case of  Randers metrics $F=\alpha+\beta$, where $\alpha=(a_{ij}(x))$ is a Riemannian metric and $\beta$ closed one-form.  It is known that a Randers metric is positive definite if and only if the Riemannian length of the vector $b_{i}(x)$ is less than one, i.e. $b(x):=\sqrt{a_{ij}(x)b^i(x)b^j(x)}< 1$, for $\forall x\in M$. 
This property also holds in the more general case of an arbitrary Randers change.

\end{enumerate}
\end{remark}

\begin{theorem}\label{theorem 1}
Let $M$ be an $n$-dimensional simply connected smooth manifold. 

A Finsler metric $F$ induces a generalized weighted quasi-distance $d_F$ on $M$ if and only if it is the Randers change of an absolute homogeneous Finsler space $F_0$ by an exact one-form $\beta$.  
\end{theorem}

\begin{proof}
We assume that $F=F_0+\beta$, where $F_0$ is an absolute homogeneous Finsler metric on $M$ and $\beta$ an exact one-form. 

Let $\gamma_{xy} \in \Gamma_{xy}$ be an $F$-geodesic, which is in the same time an $F_{0}$-geodesic, then from \eqref{Finslerian distance} we have
\begin{equation}\label{distances formula}
d_{F}(x,y)=\int_{a}^{b}F_{0}(\gamma_{xy}(t),\dot\gamma_{xy}(t))dt+
\int_{a}^{b}b_{i}(\gamma_{xy}(t))\dot\gamma_{xy}^{i}(t)dt
=d_{F_{0}}(x,y)+\int_{\gamma_{xy}}\beta.
\end{equation}

Let us consider a fixed point $a\in M$ and define the function 
$w_{a}:M\to \R,\quad w_{a}(x):=d_{F}(a,x)-d_{F}(x,a).$
From \eqref{distances formula} it follows 
\begin{equation}\label{w_a}
 w_{a}(x)=\int_{\gamma_{ax}}\beta-\int_{\gamma_{xa}}\beta=2\int_{\gamma_{ax}}\beta=
 -2\int_{\gamma_{xa}}\beta,
\end{equation}
where we have used Stokes' theorem for the one-form $\beta$ on the closed domain $D$ with boundary $\partial D:=\gamma_{ax}\cup \gamma_{xa}$.

One can see that $w_a$ is an anti-derivative of $\beta$. This is well defined if and only if the path integral in right hand side of \eqref{w_a} is path independent, that is $\beta$ must be exact. 

Then $d_{F}$ is a weighted quasi-metric with generalized weight $w_{a}$. Indeed, we have
\begin{equation}\label{rel 1}
d_{F}(x,y)+w_{a}(x)=d_{F_{0}}(x,y)+\int_{\gamma_{xy}}\beta+\int_{\gamma_{ax}}\beta-\int_{\gamma_{xa}}\beta=d_{F_{0}}(x,y)-\int_{\gamma_{xa}}\beta-\int_{\gamma_{ya}}\beta,
\end{equation}
where we have used again Stokes' theorem for the one-form $\beta$ on the closed domain with boundary $\gamma_{ax}\cup \gamma_{xy}\cup \gamma_{ya}$.

Similarly,
\begin{equation}\label{rel 2}
d_{F}(y,x)+w_{a}(y)=d_{F_{0}}(y,x)-\int_{\gamma_{ya}}\beta-\int_{\gamma_{xa}}\beta,
\end{equation}
and hence $d_{F}$ is weighted quasi-metric with generalized weight $w_{a}$.

\smallskip

Conversely, we assume that $(M,F)$ is a Finsler metric whose induced distance function $d_F$ is a weighted quasi-metric on $M$ with weight $w:M\to [0,\infty)$. For simplicity we assume that $w$ is a smooth function.

Let $\gamma:[0,\varepsilon)\to M$ be a {\it short} $C^{1}$ curve that emanates from the point
$p:=\gamma(0)\in M$ with initial velocity $v:=v^{i}\frac{\partial}{\partial x^{i}}\in T_{p}M$. Then by Busemann-Mayer Theorem (see for example \cite{BCS00}) we have
\begin{equation}\label{F(p,v)}
F(p,v)=\lim_{t\to 0^{+}}\frac{d_F(p,\gamma(t))}{t}.
\end{equation}

In a local chart $U$ around the point $p\in M$ the manifold $M$ looks locally as the Euclidean space, therefore we can write
\begin{equation}\label{F(p,-v)}
F(p,-v)=\lim_{t\to 0^{+}}\frac{d_F(\gamma(t),p)}{t},
\end{equation}
and hence, from \eqref{F(p,v)}, \eqref{F(p,-v)} and condition of weightability it results
\begin{equation}
\begin{split}
F(p,v)-F(p,-v)&=\lim_{t\to 0^{+}}\frac{d_F(p,\gamma(t))-d_F(\gamma(t),p)}{t}
=\lim_{t\to 0^{+}}\frac{w(\gamma(t))-w(p)}{t}\\
&=\frac{1}{2}\frac{\partial w}{\partial x^{i}}(p)v^{i}=dw_{p}(v).
\end{split}
\end{equation}

Then, we have
\begin{equation}
F(p,v)=\frac{1}{2}[F(p,v)-F(p,-v)]+\frac{1}{2}[F(p,v)+F(p,-v)]
=F_{0}(p,v)+\beta(p,v),
\end{equation}
where $F_0=\frac{1}{2}[F(p,v)+F(p,-v)]$, and
$\beta(p,v)=\frac{1}{2}dw_{p}(v)$. 

The geodesic reversibility is now obvious from Proposition \ref{rev_geod}.
$\qedd$
\end{proof}

\begin{remark}
Moreover, if for the arbitrary chosen point $a\in M$, there exists a constant $l_{a}$ such that
$l_{a}\leq d_{F}(a,x)-d_{F}(x,a),\quad \forall x\in M,$
then by putting $\widetilde w_{a}(x):=w_{a}(x)-l_{a}$ it follows that $(M,d_{F},\widetilde w_{a})$ is a weighted quasi-metric space. Obviously, when for example $M$ is compact, such an $l_{a}$ always exists (compare with the reversibility function used in \cite{R2004}).
\end{remark}

For later use we recall (\cite{V95}) the following lemma.

\begin{lemma}\label{quasi-metric prop}
Let $(M,d)$ be any quasi-metric space. Then $d$ is weightable if and only if there exists $w:M\to [0,\infty)$ such that 
\begin{equation}\label{formula d}
d(x,y)=\rho(x,y)+\frac{1}{2}[w(y)-w(x)], \quad \forall x,y\in M,
\end{equation}
where $\rho$ is the symmetrized distance of $d$. Moreover, we have
\begin{equation}\label{w bound}
\frac{1}{2}|w(x)-w(y)|\leq \rho(x,y), \quad \forall x,y\in M.
\end{equation}
\end{lemma}
The proof is trivial from the definition of a weighted quasi-metric.

\begin{remark}
If $(M,F)$ is a Finsler space given by the Randers change \eqref{Randers change}, then the induced quasi-metric $d_F$ and the symmetrized metric $\rho$ induce the same topology on $M$. This follows immediately from \cite{KV94}, Lemma 4. 
\end{remark}

\begin{remark}
From Lemma \ref{quasi-metric prop} it can be seen that 
the assumption of $w$ to be smooth is not essential. Indeed, from Lemma \ref{quasi-metric prop} it can be seen that if $d_F$ is a weighted quasi-metric, the function $w$ is 1-locally Lipschitz, that is differentiable almost everywhere on $M$. Therefore, the one-form $\beta$ exists almost everywhere on $M$.  
\end{remark}

\begin{remark}
See \cite{M12} for a very interesting discussion on the completeness of a Randers change by means of an exact one-form $\beta$. 
\end{remark}

We discuss an interesting geometric property concerning the geodesic triangles.

\begin{proposition}\label{Prop 1.6}
Let $(M,F)$ be a Finsler metric given by the Randers change  \eqref{Randers change}. Then the perimeter length of any geodesic triangle on $M$ does not depend on the orientation, that is
\begin{equation}\label{perimeter}
d_F(x,y)+d_F(y,z)+d_F(z,x)=d_F(x,z)+d_F(z,y)+d_F(y,x),\quad \forall x,y,z\in M.
\end{equation}
\end{proposition}

\bigskip

\bigskip

\bigskip

\setlength{\unitlength}{1cm} 
\begin{center}
\begin{picture}(5, 3)

\put(3,1){\circle*{0.08}}
\put(2.8,0.7){$x$}
\qbezier(3,1)(3,3)(4,4)
\put(4,4){\circle*{0.08}}
\qbezier(3,1)(4,3)(4.94,3.4)
\qbezier(4,4)(4.9,4)(4.94,3.4)
\put(4.94,3.4){\circle*{0.08}}
\put(3.7,4){$y$}
\put(3.35,3){\vector(-1,-2){0.04}}
\put(4.35,3){\vector(1,1){0.04}}
\put(4.5,3.94){\vector(-3,1){0.04}}
\put(5,3.2){$z$}

\put(0,1){\circle*{0.08}}
\put(-0.2,0.7){$x$}
\qbezier(0,1)(0,3)(1,4)
\put(1,4){\circle*{0.08}}
\qbezier(0,1)(1,3)(1.94,3.4)
\qbezier(1,4)(1.9,4)(1.94,3.4)
\put(1.94,3.4){\circle*{0.08}}
\put(0.7,4){$y$}
\put(0.35,3){\vector(1,2){0.04}}
\put(1.35,3){\vector(-1,-1){0.04}}
\put(1.5,3.94){\vector(3,-1){0.04}}
\put(2,3.2){$z$}

\end{picture}
\end{center}
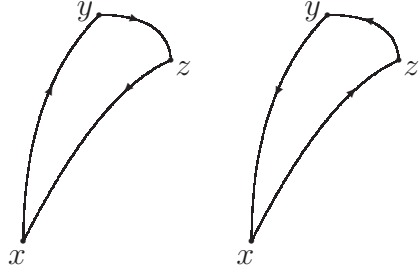
{\bf Figure 2.} The perimeter of the triangle $\Delta xyz$ is independent of the orientation.

\bigskip

In other words, even though the distance between two points $x$ and $y$ depends on the orientation of a minimizing geodesic joining points $x$ and $y$, i.e. $d_F(x,y)\neq d_F(y,x)$, the sum of distances between three points $x$, $y$, $z$ on $M$ do not depend on the direction we trace out the perimeter of the geodesic triangle $\Delta xyz$. We point out that weighted quasi-metric spaces can be characterized by this property without the explicit use of the weight function. Indeed, a  quasi-metric $d$ is weightable if and only if relation \eqref{perimeter} holds.

\begin{proof}

The proof is almost trivial. Since $F$ it is the Randers change  \eqref{Randers change}, then 
 from Theorem \ref{theorem 1} it follows that the quasi-metric is weightable and therefore  \eqref{formula d} holds good. By using this formula an elementary computation proves \eqref{perimeter}.
$\qedd$
\end{proof}

Moreover, we have

\begin{proposition}
Let $(M,F)$ be a Finsler space that satisfies \eqref{perimeter}. Then $F$ can be written as the  Randers change of an absolute homogeneous Finsler metric $F_0$ by an exact one-form $\beta$.
\end{proposition}
\begin{proof}

It is easy to see that if a quasi-metric satisfies  \eqref{perimeter} then it is weightable (see \cite{V99}). Then conclusion follows from Theorem \ref{theorem 1}.
$\qedd$
\end{proof}

\begin{remark}
It should be clear that not any quasi-metric space is weightable. In fact, it can be shown that the class of weightable quasi-metric spaces are exactly those quasi-metric spaces that satisfy relation \eqref{perimeter} (see \cite{V99}).
\end{remark}

\section{Isometric embeddings of Finsler spaces}

If $(X,q,w)$ and $(Y,p,u)$ are two weighted quasi-metric spaces, the mapping $\varphi:X\to Y$ with the properties
\begin{eqnarray}\label{metric ineq}
& & p(\varphi(x),\varphi(y))\leq q(x,y),\quad \forall x,y\in X\\ 
& & u(\varphi(x))\leq w(x), \quad \forall x\in X \label{weight ineq}
\end{eqnarray}
is called a {\it morphism} of weighted quasi-metric spaces. 

In the case we have equality in relation \eqref{metric ineq}, then the morphism $\varphi$ is called an {\it isometric morphism}. In this case $w$ and $u \circ \varphi$ differ by a constant only. 

Moreover, an {\it isomorphism} of the weighted quasi-metric spaces  $(X,q,w)$ and $(Y,p,u)$ is a bijective function  $\varphi:X\to Y$ that preserves both the quasi-metric and the weight function.

Finally, an {\it embedding} of  $(X,q,w)$ into $(G,Q,W)$ is an isomorphism of  $(X,q,w)$ onto a subspace of  $(G,Q,W)$. Here, a {\it subspace $(Y,p,u)$ of a weighted quasi-metric 
space} $(G,Q,W)$ is a subset $Y\subset G$, the function $p$ and $u$ are the restriction of $Q$ and $W$ to $Y\times Y$ and $Y$, respectively. 

\begin{example}[The product of a metric space with a half ray]\label{space times ray}
Consider a metric space $(S,d)$ and the half ray $I:=[0,\infty)$. Then the product space 
$
G:=S\times I
$
inherits a natural structure of (generalized) weighted quasi-metric space $(G,Q,W)$, where
\begin{equation}\label{Q,W for space times ray}
\begin{split}
& Q:G\times G\to [0,\infty), \quad Q(u,v):=d(x,y)+\eta-\xi,\\
& W:G\to [0,\infty),\quad W(u):=2\xi,\quad \forall u=(x,\xi),v=(y,\eta)\in S\times I.
\end{split}
\end{equation}
\end{example}

\begin{remark}
The generalized weighted quasi-metric space $(S\times I,Q,W)$ constructed in Example \ref{space times ray} is sometimes called {\it the bundle over $(S,d)$} (see \cite{V99}).
\end{remark}


\begin{example}[The Graph of a function]\label{Graph of a function}
We consider the case of the graph of a non-negative valued function $f:S\to [0,\infty)$ defined on a metric space $(S,d)$. 

Indeed, if we denote the graph of $f$ by 
$G_f:=\{(x,f(x)):x\in S\}$
then $(G_f,Q,W)$ is a naturally induced  weighted quasi-metric space structure defined by
\begin{equation}
\begin{split}
& Q:G_f\times G_f\to [0,\infty), \quad Q(u,v):=d(x,y)+f(y)-f(x),\\
& W:G_f\to [0,\infty),\quad W(u):=2f(x), \quad \forall u=(x,f(x)),v=(y,f(y))\in G_f.
\end{split}
\end{equation}
\end{example}

Based on these, one has
\begin{theorem}[\cite{V99}]\label{embedding theorem}
Every  weighted quasi-metric space $(X,q,w)$ is embeddable in a bundle over a suitable metric space $(S,d)$. 
\end{theorem}
The idea of the proof is simple. 
Following Example \ref{space times ray}, given the weighted quasi-metric space $(X,q,w)$ one constructs a naturally associated product of a metric space $(S,d)$ and a half line.

The obvious choice for $(S,d)$ is the symmetrization of  the quasi-metric space 
$(X,q)$.
Therefore one has the natural weighted quasi-metric space $(G,Q,W)$, where 
 $G:=X\times [0,\infty)$, $Q$ and $W$ are defined in \eqref{Q,W for space times ray}.

One defines now the function 
$\varphi:X\to G, \ \varphi(x):=(x,\frac{1}{2}w(x)),$
and show that this is indeed an embedding. 

A fundamental result is that any weighted quasi-metric space can be constructed starting from a metric space $(S,d)$  and a 1-Lipschitz function $f:S\to [0,\infty)$ defined on it, i.e.
\begin{equation}\label{Lipschitz condition}
|f(x)-f(y)|\leq d(x,y), \quad \forall x,y\in S.
\end{equation}

\begin{theorem}[\cite{V99}]\label{induced constr}
\begin{enumerate}
\item Let $(S,d)$ be a metric space and $f:S\to [0,\infty)$ a 1-Lipschitz function. Then the graph of $f$ is a weighted quasi-metric space $(G_f,Q,W)$.
\item Conversely, every weighted  quasi-metric space $(X,q,w)$ can be constructed in this way.
\end{enumerate}
\end{theorem}

The proof is also quite obvious. Statement 1  is straightforward from Example  \ref{Graph of a function}. We point out that Lipschitz condition \eqref{Lipschitz condition} guaranties that 
$Q(u,v)\geq 0$, i.e.  $(G_f,Q,W)$ is actually a weighted  quasi-metric space.

Statement 2 follows from proof of Theorem \ref{embedding theorem}. Indeed, given a  weighted  quasi-metric space $(X,q,w)$ one can construct 
\begin{itemize}
\item a metric space $(S,d):=(X,\rho)$, where $\rho$ is the symmetrization of $q$,
\item a Lipschitz function $f:S\to [0,\infty)$, $f(x):=\frac{1}{2}w(x)$.
\end{itemize}

One can see that this $f$ always satisfies the Lipschitz condition \eqref{Lipschitz condition} on $(S,d)$ because of \eqref{w bound}.
Moreover, due to  Theorem \ref{embedding theorem} there is an embedding of  $(X,q,w)$ onto  $(G_f,Q,W)$. That is recover the original  weighted  quasi-metric space $(X,q,w)$ from $(G_f,Q,W)$ by identifying  $X$ with a subspace of  $G$ obtained by the obvious projection and restricting $Q$ and $W$ to this $X$.


Next, we recall the differential manifold structure of the graph of a smooth
 function.

Let us consider a $C^{\infty}$ function 
$
f:M\to [0,\infty),\ x\mapsto f(x)
$
and the graph of $f$ denoted by $G_{f}=\{(x,f(x)):x\in M\}\subset M\times \R$. Then it is known that $G_{f}$ is a $C^{\infty}$ submanifold of the product manifold $M\times \R$ that is actually diffeomorphic to $M$. Indeed, the mapping
\begin{equation}
\varphi:M\to G_f,\quad x\mapsto \varphi(x)=(x,f(x))
\end{equation}
with the inverse
\begin{equation}\label{inverse psi}
\psi:G_f\to M,\quad u=(x,f(x))\mapsto \psi(x,f(x))=x
\end{equation}
is a diffeomophism. Remark that $\psi$ is nothing else than the projection onto the first factor. 

Any given weighted quasi-metric space $(M,q,w)$ that satisfies some supplementary metrizability condition induced a Finsler structure $(M,F=F_{0}+df)$ on $M$ and conversely, every given weighted quasi-metric space $(M,q,w)$ can be constructed in this way.

We recall the smooth approximation of Lipschitz functions on a Finsler manifold:
\begin{lemma}[\cite{M12}]\label{smooth approx}
Let $(M,F)$ be a Finsler manifold and $f:M\to \R$ a 1-Lipshitz function, i.e.
\begin{equation}
|f(x)-f(y)|\leq d_{F}(x,y), \quad \forall x,y\in M,
\end{equation}
where $d_{F}$ is the Finslerian induced quasi-distance on $M$. Then, for any small positive $\varepsilon_{1}$, $\varepsilon_{2}$ there exists a smooth function $\widetilde{f}:M\to \R$ such that
\begin{enumerate}
\item $|\widetilde{f}(x)-f(x)|<\varepsilon_{1}$,  $\forall x\in M$,
\item $\widetilde{f}$ is $(1+\varepsilon_{2})$-Lipschitz, i.e. 
$|\widetilde{f}(x)-\widetilde{f}(y)|\leq (1+\varepsilon_{2})d_{F}(x,y), \quad \forall x,y\in M.$
\end{enumerate}
\end{lemma}

The Riemannian version of this lemma can be found in \cite{A07}.

Then, we have

\begin{theorem}
\begin{enumerate}
\item Let $(M,F=F_{0}+df)$ be a Finsler space, where $f$ is a $C^{\infty}$ non-negative function $f$ on $M$. Then the graph manifold $G_{f}$ inherits a natural structure of weighted quasi-metric space $(G_{f},Q,W)$ that coincides with the weighted quasi-metric space $(M,d_{F}, 2f)$ induced by $F$, up to an isomorphism. 
\item For any given weighted quasi-metric space $(M,q,w)$, whose symmetrized metric is $C^{\infty}$-Riemannian (or absolutely homogeneous Finsler) metrizable, there exist
	\begin{enumerate}
	\item  a smooth approximation function $\widetilde{f}:M\to [0,\infty)$ 
of $w$,
	\item a weighted quasi-metric space $(G_{\widetilde f},\widetilde{Q}, \widetilde{W})$, called the  smooth approximation of  $(G_{f},Q,W)$, that coincides with the weighted quasi-metric space $(M,d_{F}, 2\widetilde{f})$ induced by a (not necessarily positive definite) Randers metric 
	$F=\widetilde\alpha+d\widetilde{f}$ (or Randers change $F=F_{0}+d\widetilde{f}$), up to an isomorphism. 
	\end{enumerate}
\end{enumerate}
\end{theorem}

\begin{proof}
1. If we start with the Finsler structure $(M,F=F_{0}+df)$, then this is with reversible geodesics, and therefore from Theorem \ref{theorem 1} it follows that $M$ becomes a weighted quasi-metric space $(M,d_{F}, 2f)$, where
$d_{F}(x,y)=d_{F_{0}}(x,y)+f(y)-f(x),\quad \forall x,y\in M.$

The symmetrized metric $\rho$ of $d_{F}$ coincides with $d_{F_{0}}$ 
and therefore the graph manifold $G_{f}$ becomes a weighted quasi-metric space $(G_{f},Q,W)$, where 
\begin{equation}
 Q(u,v)=\rho(x,y)+f(y)-f(x),\ W(u)=2f(x),\quad \forall u=(x,f(x)),v=(y,f(y))\in G_{f}.
\end{equation}

It can be easily seen that this $(G_{f},Q,W)$ is indeed a weighted quasi-metric space and that 
$\psi:(G_{f},Q,W)\to (M,d_{F}, 2f)$ defined in \eqref{inverse psi} is an isometric embedding.

\bigskip

2. Starting with an arbitrary weighted quasi-metric space $(M,q,w)$ remark that the weight $w:M\to [0,\infty)$ is a 1-Lipschitz function (see \eqref{w bound}) with respect to the symmetrized metric $\rho$. This is not good enough to define a $C^{\infty}$-Finsler metric because $w$ has a measure zero set of points where it fails to be differentiable.

We are going to use the smooth approximation of Lipschitz functions on Riemannian manifolds (\cite{A07}) or Finsler manifolds (see Lemma \ref{smooth approx}). For the sake of simplicity we present only the Riemannian case here. We put $f:=\frac{1}{2}w$ and denote the smooth approximation of $f$ by $\widetilde f$. It follows that $G_{\widetilde f}$ is a smooth manifold that inherits a natural structure of weighted quasi-metric space from $(G_{f},Q,W)$ constructed in Theorem \ref{induced constr}, 1. Namely, we put
\begin{equation}
 \widetilde Q(u,v):=(1+\varepsilon_{2})\rho(x,y)+\widetilde f(y)-\widetilde f(x),\
 \widetilde W(u):=2\widetilde f(x),\ \forall u=(x,\widetilde f(x),  v=(y,\widetilde f(y))\in G_{\widetilde f}.
\end{equation}
Since $\widetilde f$ is $(1+\varepsilon_{2})$-Lipschitz with respect to $\rho$ it follows that $\widetilde Q(u,v)\geq 0$ for any $u,v\in G_{\widetilde f}$. Elementary computations shows that indeed $(G_{\widetilde f},\widetilde{Q}, \widetilde{W})$ is a weighted quasi-metric space that smoothly approximates $(G_{f},Q,W)$ in the sense that 
\begin{equation}
 |\widetilde Q(u,v)-Q(u,v)|\leq 2\varepsilon_{1},\quad
 |\widetilde W(u)-W(u)|\leq 2\varepsilon_{1}, \quad \forall u,v\in G_{\widetilde f}.
\end{equation}

We define now $\widetilde a_{ij}(x):=(1+\varepsilon_{2})^{2}a_{ij}(x)$, where $(M,a)$ is the Riemannian metric corresponding to the metric  space $(M,\rho)$, for $\forall i,j\in\{1,2,\dots,n\}$, and $\forall x\in M$. 

The Randers space $(M,F=\widetilde\alpha+d\widetilde f)$ induces a structure of weighted quasi-metric space $(M,d_{F}, 2\widetilde{f})$ on $M$ which is isometrically embeddable into $(G_{\widetilde f},\widetilde{Q}, \widetilde{W})$ as shown above. Obviously, this Randers space is positive definite if and only if the Riemannian length of the gradient vector $grad\  \widetilde f$ is less than one. 

The proof is identical if $(M,\rho)$ is absolutely homogeneous Finsler metrizable. 
$\qedd$
\end{proof}

\begin{corollary}
For any given weighted quasi-metric space $(M,q,w)$, whose symmetrized metric is $C^{\infty}$-Riemannian (or absolutely homogeneous Finsler) metrizable, and whose weight $w$ is a smooth function, 
the weighted quasi-metric space $(M,d_{F}, 2{f})$ induced by the Randers metric $F=\alpha+d{f}$ 
(or Randers change $F=F_{0}+d{f}$)
coincides with $(M,q,w)$.
\end{corollary}

\begin{remark}
Obviously not any metric space $(M,d)$ is Riemannian metrizable. Here Riemannian metrizable means that $M$ is a differentiable manifold and that there exists a 
$C^\infty$-Riemannian metric $a=a_{ij}(x)$ on $M$ whose associated distance function $d_{\alpha}$ coincides with $d$. 

General conditions for a metric space to be Riemannian metrizable can be found in \cite{N99}. Similar conditions can be easily established for metric spaces to be absolute homogeneous Finsler metrizable. 

More generally one can study conditions for quasi-metric space to be Finsler metrizable. Metrizability of metric or quasi-metric spaces is a complex subject that we intend to discuss in a forthcoming paper. 
\end{remark}

\bigskip

We discuss now another representation of Finsler spaces. 

We recall that for a metric space the {\it Hausdorff distance} is a distance function between subsets of $M$. Indeed, if $(M,d)$ is a metric space, then the mapping
\begin{equation}
d_H:2^M\times 2^M\to [0,\infty),\quad d_H(A,B)=\max\{\sup_{a\in A}d(a,B),
\sup_{b\in B}d(b,A)\}
\end{equation}
is called the {\it Hausdorff metric}, where $2^M$ is the set of all subsets of $M$.

In the case of $2^M$ the function $d_H$ is only a semi-metric. However, the pair $(\mathcal{P}_0(M,d),d_H)$ is a metric space, where $\mathcal{P}_0(M,d)$ is the set of non-empty closed subsets of $M$ (\cite{BBI01}).

This notion can be easily extended to the more general case of quasi-metric spaces as follows. Let $(M,q)$ be a quasi-metric space.
We define the mappings:
\begin{equation}
\begin{split}
& d_H^f(A,B):=\sup_{a\in A}q(a,B),\quad 
 d_H^b(A,B):=\sup_{b\in B}q(A,b),\\
& d_H(A,B):=\max\{d_H^f(A,B),d_H^b(A,B)\},\quad \forall A,B\in \mathcal{P}_0(M,d).
\end{split}
\end{equation}

It can be seen that $d_H^f$, $d_H^b$, $d_H$ are extended quasi-metrics on  $\mathcal{P}_0(M,q)$ and that they become quasi-metrics when restricted to 
$\mathcal K_0(M,q)$, that is the set of all non-empty compact subsets of $(M,q)$. 

The pair $( \mathcal K_0(M,q),d_H)$ is called the associated 
{\it quasi-Hausdorff metric} of $(M,q)$ (compare with \cite{S01}).

Many of the geometrical properties of Hausdorff distance extend to the case of quasi-Hausdorff distances (a detailed study of these together with the Gromov-Hausdorff distance will be given elsewhere). 

Let ($M,q)$ be a quasi-metric space, and construct the metric space $(X,\delta)$, where
$X:=M\times [0,\infty)$, and
$\delta((x,\xi),(y,\eta))=d^*(x,y)+|\xi-\eta|,$
for $\forall (x,\xi),(y,\eta)\in X$. 
Here $d^*$ is the max-metric \eqref{max metric}.

It can be easily seen that indeed $(X,\delta)$ is a metric space and that for $\forall z\in M$, the set
$E(z):=\{(y,\eta)\in X:d(y,z)\leq \eta\}$
is a non-empty closed subset in  $(X,\delta)$, i.e. $E(z)\in \mathcal{P}_0(X,\delta)$. We write here $\delta$ in order to make explicit the topology where the set are closed. 

We recall (\cite{V95}) that
if $(M,d)$ is a quasi-metric space, then the mapping
$E:(M,d)\to (\mathcal{P}_0(X,\delta),d_H^f), \ z\mapsto E(z)$
is an isometry of $M$ onto a subspace $E(M)$ of $\mathcal{P}_0(X,\delta)$.
Indeed, it can be seen that 
$E$ is injective, and
$d(x,y)=d_H^f(E(x),E(y))$, for $\forall x,y\in M$.

We obtain

\begin{proposition}
Let $(M,F)$ be a Finsler space with associated quasi-metric $d_F$. Then the quasi-metric space $(M,d_F)$ is isometric to a subspace $E(M)$ of $(\mathcal{P}_0(M),d_H^f)$.
\end{proposition}

Let us assume now that $d_F$ is a weighted quasi-metric. In this case we have
$d_H^f(A,B)=d_H^\rho(A,B)+\frac{1}{2}\Bigl(\inf_{b\in B}w-\sup_{a\in A}w
\Bigr),$
where we have used
$d(a,B)=\inf_{b\in B}d(a,b)=\rho(a,B)+\frac{1}{2}\Bigl(\inf_{b\in B}w-w(a)
\Bigr).$
Here $d_H^\rho(A,b)$ is the usual Hausdorff distance of the symmetrized metric $\rho$. 

It can be seen that the forward Hausdorff distance can not exceed the $\rho$-Hausdorff distance.



%
%
%
%
\end{document}